\documentclass[a4paper,12pt]{article}
\usepackage[utf8]{inputenc}
\usepackage{amsmath,amsthm,amssymb}
\usepackage{graphicx}
\usepackage{navigator}
\embeddedfile{sourcefile}{null.tex}

\makeatletter
\def\@seccntformat#1{\csname the#1\endcsname.\ } % the column after a section number
\makeatother

\title{On minimal subspace $Z_p$-null designs%
\thanks{The work was funded by the Russian Science Foundation under grant 18-11-00136}}
\author{Denis S. Krotov}
\date{}

\newtheorem{theorem}{Theorem}
\newtheorem{proposition}[theorem]{Proposition}

\newtheorem{corollary}[theorem]{Corollary}
\theoremstyle{remark}

\newcommand\FF{\mathbb{F}}
\newcommand\ZZ{\mathbb{Z}}
\newcommand\Cc{\mathcal{C}}
\newcommand\Zq{\ZZ_q}
\newcommand\Zr{\ZZ_r}

\newcommand\Fq{\FF_q}

\newcommand\Jq[2][q]{J_{#1}(n,#2)}
\newcommand\JJq[2][q]{\Jq[#1]{{\ge}#2}}
\newcommand\vv[1]{\boldsymbol{#1}}

\begin{document}

\maketitle

\begin{abstract}
Let $q$ be a power of a prime $p$,
and let $V$ be an $n$-dimensional space
over the field GF$(q)$.
A $Z_p$-valued function $C$ on the 
set of $k$-dimensional subspaces of $V$
is called a $k$-uniform 
$Z_p$-null design of strength $t$
if for every $t$-dimensional subspace $y$ of $V$
the sum of $C$ over the $k$-dimensional superspaces 
of $y$ equals $0$.
For $q=p=2$ and $0\le t<k<n$, 
we prove that the minimum number of non-zeros
of a non-void $k$-uniform $Z_p$-null design of strength $t$
equals $2^{t+1}$. 
For $q>2$, we give lower and upper bounds 
for that number.
\end{abstract}

\section{Introduction}

Let $p$ be a prime number, 
and let $q=p^s$ where $s$ is a positive integer. 
We denote by $\Fq$ the finite field GF$(q)$ of order $q$
and by $\ZZ_s$ the factor ring $\ZZ / s \ZZ$.
By $\Jq{k}$, 
we denote the set of $k$-dimensional subspaces 
of the $n$-dimensional space $\Fq^n$ over $\Fq$.
We will also use the notation 
$$ \JJq{ t} := \bigcup_{k=t}^n \Jq{k}.  $$
Let $\mathbb{G}$ be an abelian group.
A function
$\Cc: \JJq{ t} \to \mathbb{G}$
is called a 
\emph{subspace $\mathbb{G}$-null design of strength $t$}
(depending on the context, 
we will omit ``subspace'', ``$\mathbb{G}$-'', 
and/or ``of strength $t$'') 
if for every $\vv{y}$ from $\Jq{t}$ we have
$\Cc({\ge}\vv{y})=0$, where
$$\Cc({\ge}\vv{y}):= \sum_{\vv{x}\in \JJq{t}:\, \vv{y}\subset \vv{x}} \Cc(\vv{x}). $$
The identity-zero null design is called \emph{void}.
A null design is called \emph{$k$-uniform} or simply \emph{uniform}
if all its nonzeros lie in $\Jq{k}$. 
Our main goal is to study
 the minimum number of non-zeros
of uniform null designs; 
however, 
it is convenient for some statements 
to be proved in a general form, 
for unrestricted null designs.

The following treatment of the strength 
is well known for many kinds of designs; 
however, usually it holds only for uniform designs 
but not in the unrestricted case in general.
So, it is notable that we do not need the uniformity in the following
proposition.

\begin{proposition}\label{p:tt-1}
Let $q$ and $r$ be powers of a prime $p$ such that 
$p \le r \le q$.
If $\Cc:\,\JJq{t}\to \Zr$
satisfies $\Cc({\ge}\vv{z})=\lambda$ for all
$\vv{z}$ from $\Jq{t}$ and some constant 
$\lambda$ from $\Zr$,
then $\Cc$ 
(continued by zeros to the extended domain)
satisfies $\Cc({\ge}\vv{y})=\lambda$ 
for all
$\vv{z}$ in $\Jq{\tau}$, $\tau \in \{ 0,\ldots,t \}$.
In particular, if $\Cc$ is a $\Zr$-null design of strength $t$,
then it is 
a $\Zr$-null design of strength $\tau$ 
for any $\tau \in \{ 0,\ldots,t \}$.
\end{proposition}
\begin{proof}
 For $\vv{y}$ from $\Jq{t-1}$ and 
 $\vv{x}$ from $\JJq{ t}$, we note that
\begin{multline}\label{eq:modr}
 |\{\vv{z}\in \Jq{t}:  \vv{y}\subset \vv{z} \subset \vv{x} \}| 
 = \frac{q^{\dim(\vv{x})-t+1}-1}{q-1}
 \\
 = q^{\dim(\vv{x})-t}+q^{\dim(\vv{x})-t-1}+ \ldots + q^{0} \equiv 1 \bmod r.
\end{multline}
 So, for every $\vv{y}$ from $\Jq{t-1}$ we have 
 \begin{multline*}
   \sum_{\rotatebox{2}{$\scriptstyle\vv{x}\in \JJq{ t}:\,  \vv{y}\subset \vv{x}$}} \Cc(\vv{x})
    \stackrel{\eqref{eq:modr}}{=}
    \sum_{\rotatebox{2}{$\scriptstyle\vv{x}\in \JJq{ t}:\,  \vv{y}\subset \vv{x}$}} \ 
    \sum_{\rotatebox{2}{$\scriptstyle\vv{z}\in \Jq{t}:\,  \vv{y}\subset \vv{z} \subset \vv{x}$}} \Cc(\vv{x}) \\
    =
    \sum_{\rotatebox{2}{$\scriptstyle\vv{z}\in \Jq{t}:\,  \vv{y}\subset \vv{z}$}} \ 
    \sum_{\rotatebox{2}{$\scriptstyle\vv{x}\in \JJq{ t}:\,  \vv{z}\subset \vv{x}$}} \Cc(\vv{x})
    =
    \sum_{\rotatebox{2}{$\scriptstyle\vv{z}\in \Jq{t}:\,  \vv{y}\subset \vv{z}$}} \lambda
    \stackrel{\eqref{eq:modr}}{=}
    \lambda,
 \end{multline*}
and the claim is true for $\tau=t-1$.
By induction, it holds for  
$\tau \in \{ 0,\ldots,t \}$.
\end{proof}

\section{Main results}

\begin{theorem}\label{th:lb}
For any abelian group $\mathbb{G}$,
the minimum number of non-zeros 
in a non-void $\mathbb{G}$-null design $\Cc:\JJq{ t}\to \mathbb{G}$ 
of strength $t$, $0\le t<n$, is $1+\frac{q^{t+1}-1}{q-1}$.
\end{theorem}
\begin{proof}
Let $e$ be any non-identity element of $\mathbb{G}$, and let $-e$ be its inverse.
If $\Cc(\vv{v})=e$ for some $\vv{v}$ from $\Jq{t+1}$, 
$\Cc(\vv{u})=-e$ for every $t$-dimensional subset $\vv{u}$ of $\vv{v}$,
and $\Cc(\vv{w})=0$ for any other $\vv{w}$  from $\JJq{ t}$,
then $\Cc$ is a null design of strength $t$ with 
$1+\frac{q^{t+1}-1}{q-1}$ non-zeros.

By backward induction on $t$, we will prove that 
\textit{the number of non-zeros
of non-void null design of strength $t$ is not less than
$1+\frac{q^{t+1}-1}{q-1}$}. 
For $t=n$, the claim is trivial because there are no non-void null designs.
Assume that the claim is true for $t=\tau$, where $0<\tau\le n$; 
let us show it for $t=\tau-1$. Consider a non-void null design $\Cc$
of strength $t$. Denote 
$\Cc'=\Cc\big|_{\JJq{t{+}1}}$.
Clearly, $\Cc'$ is not constantly zero (otherwise, $\Cc$ is void as well).
If $\Cc'$ is a null design of strength $t+1$, 
then the claim is true by the induction hypothesis. 
Otherwise, there is $\vv{v}$ in $\Jq{t+1}$ such that $\Cc({\ge} \vv{v})\ne 0$.
% , where
% $$
% \Cc({\ge} \vv{v}) := \sum_{\vv{w} \in \JJq{t{+}1}:\,\vv{v}\subset\vv{w}}.
% $$
By the definition of null design, for every $t$-dimensional subspace $\vv{u}$ of $\vv{v}$,
we have $\Cc({\ge} \vv{u}) = 0$. So, the sum of $\Cc$ over $U_{\vv{u}}$, 
where $$U_{\vv{u}}:=\{\vv{w} \in \JJq{t{+}1}:\,\vv{u}\subset\vv{w},\vv{v}\not\subset\vv{w}\},$$
equals $-\Cc({\ge} \vv{v})$. Hence, $U_{\vv{u}}$ contains a non-zero of $\Cc$.
Since $U_{\vv{u}}$ and $U_{\vv{u'}}$ are disjoint for different $t$-dimensional subspaces 
$\vv{u}$ and $\vv{u}'$ of $\vv{v}$, we have at least $1+\frac{q^{t+1}-1}{q-1}$ non-zeros.
\end{proof}
The null design constructed in the proof is not uniform. So, for the number of non-zeros in a non-void uniform null design,
Theorem~\ref{th:lb} only gives a lower bound.

\begin{theorem}\label{th:up}
For every $k$, $t<k<n$,
there is a $k$-uniform % $\{0,1\}$-valued 
$\Zq$-null design $\Cc:\Jq{k}\to \{0,1\}$ 
of strength $t$ with $q^{t+1}$ non-zeros.
\end{theorem}

Note that a $\{0,1\}$-valued $\Zq$-null design is a $\Zr$-null design
for every $r$, $r|q$.

\begin{proof}
% Denote $d=k-t$. 
Take any $\vv{w}$ from $\Jq{k+1}$,
$\vv{u}$ from $\Jq{k-t-1}$, 
and $\vv{v}$ from $\Jq{k-t}$ 
such that 
$\vv{u}\subset \vv{v}\subset \vv{w}$.
Denote $U := \{ \vv{x} \in  \Jq{k}:\, \vv{u}\subset \vv{x}\subset \vv{w} \}$
and $V := \{ \vv{x} \in  \Jq{k}:\, \vv{v}\subset \vv{x}\subset \vv{w} \}$.
We state that the characteristic $\{0,1\}$-function of 
$U\backslash V$ is a required null design.

At first, we note that $|U\backslash V| = |U|-|V| = q^{t+1}$
because
$$|U| = \frac{q^{(k+1)-(k-t-1)}-1}{q-1} = \frac{q^{t+2}-1}{q-1} 
\quad \mbox{and }
|V| = \frac{q^{(k+1)-(k-t)}-1}{q-1} = \frac{q^{t+1}-1}{q-1}.$$

Next, consider $\vv{y}$ from $\Jq{t}$ and count the numbers
$$
\alpha = |\{ \vv{x} \in U :\, \vv{y}\subset \vv{x} \} | 
\quad \mbox{and }
\beta = |\{ \vv{x} \in V :\, \vv{y}\subset \vv{x} \} |.
$$
If $\vv{y} \not\subset \vv{w}$, then obviously $\alpha=\beta=0$.
If $\vv{y} \subset \vv{w}$, then we have
$$ 
\alpha = \frac { q^{(k+1)-\dim(\vv{u}+\vv{y})} - 1 } {q-1} \equiv 1 \bmod q
\quad \mbox{and }
\beta = \frac { q^{(k+1)-\dim(\vv{v}+\vv{y})} - 1 } {q-1} \equiv 1 \bmod q. $$
Anyway, $\alpha - \beta \equiv 0 \bmod q$.
Hence, $U\backslash V$ is a $\Zq$-null design of strength $t$.
\end{proof}
For $q=2$, the lower and upper bounds in Theorems~\ref{th:lb}
and~\ref{th:up} coincide.

\begin{corollary}\label{c:2}
 For $0\le t <k<n$,
the minimum number of non-zeros 
in a non-void $k$-uniform $\ZZ_2$-null design 
$\Cc:\JJq[2]{t}\to \ZZ_2$ 
of strength $t$ is $2^{t+1}$.
\end{corollary}

\section{On Wilson matrices}
The Wilson matrix $W_{t,k}$ is a $\{0,1\}$-matrix
whose rows are indexed by $t$-subsets of $\{1,\ldots,n\}$
and
columns are indexed by $k$-subsets of $\{1,\ldots,n\}$;
the element of $W_{t,k}$
in the $\vv{y}$th row and $\vv{x}$th column
equals $1$ if and only if 
$\vv{y} \subset \vv{x}$.
By analogy, the subspace Wilson matrix $W_{q;t,k}$ is a $\{0,1\}$-matrix
whose rows and
columns are indexed by $\Jq{t}$
and $\Jq{k}$, respectively;
$W_{q;t,k}(\vv{y},\vv{x})=1$ if and only if 
$\vv{y} \subset \vv{x}$.
For the case $q=2$, we have established (Corollary~\ref{c:2}) that 
the binary linear
code with check matrix $W_{q;t,k}$ 
has minimum distance $2^{t+1}$.
(A similar result for $W_{k-1,k}$ 
was obtained in~\cite[Prop.~9]{Pot:DP}.)
However, the dimension of this code 
(equivalently, the rank of $W_{q;t,k}$)
remains unknown for $t>1$.
In the case $t=1$, $W_{q;1,k}^{\mathrm{T}}$ is a generator matrix 
of the punctured Reed--Muller code $\mathcal{R}^*(n-k,n)$,
and
$\mathrm{rank}(W_{q;1,k})=\sum_{i=0}^k\binom{n}{i}$, see e.g.~\cite[Theorem~3.14]{KeyAss:98}.
Much less is known for $q>2$. It is also known that 
$W_{q;t,k}$ is full-rank over $\mathbb{R}$, see \cite{Kantor:72}; 
the minimum 
number of non-zeros in a non-void $\mathbb{R}$-null design
$\Jq{k} \to \mathbb{R}$ is $\prod_{i=0}^t(1+q^i)$
if $k=t+1$ and is conjectured to be $\prod_{i=0}^t(1+q^i)$
if $t+2\le k < n-t$, see \cite{Cho:98} 
(with the additional restriction on the null design to be $\{0,1,-1\}$-valued, the conjecture was proved in~\cite{Kro:17:subspace}).

%   \bibliographystyle{plain}
%   \bibliography{../../k}

\begin{thebibliography}{1}

\bibitem{KeyAss:98}
Jr. Assmus, E.~F. and J.~D. Key.
\newblock Polynomial codes and finite geometries.
\newblock In V.S. Pless and W.~C. Huffman, editors, {\em Handbook of Coding
  Theory}, chapter~16, page 1269–1343. Elsevier, New York, 1998.

\bibitem{Cho:98}
S.~Cho.
\newblock Minimal null designs of subspace lattice over finite fields.
\newblock {\em
  \href{http://www.sciencedirect.com/science/journal/00243795}{Linear Algebra
  Appl.}}, 282(1-3):199--220, 1998.
\newblock \DOI{10.1016/S0024-3795(98)10062-9}.

\bibitem{Kantor:72}
W.~M. Kantor.
\newblock On incidence matrices of finite projective and affine spaces.
\newblock {\em
  \href{http://link.springer.com/journal/volumesAndIssues/209}{Mathematische
  Zeitschrift}}, 124(4):315--318, Dec. 1972.
\newblock \DOI{10.1007/BF01113923}.

\bibitem{Kro:17:subspace}
D.~S. Krotov.
\newblock The minimum volume of subspace trades.
\newblock {\em
  \href{http://www.sciencedirect.com/science/journal/0012365X}{Discrete
  Math.}}, 340(12):2723--2731, 2017.
\newblock \DOI{10.1016/j.disc.2017.08.012}.

\bibitem{Pot:DP}
V.~N. Potapov.
\newblock Splitting of hypercube into $k$-faces and {DP}-colorings of
  hypergraphs.
\newblock E-print 1905.04461v4, arXiv.org, 2020.
\newblock Available at \url{http://arxiv.org/abs1905.04461v4}.

\end{thebibliography}
%   \end{document}

\providecommand\href[2]{#2} \providecommand\url[1]{\href{#1}{#1}}
  \def\DOI#1{{\small {DOI}:
  \href{http://dx.doi.org/#1}{#1}}}\def\DOIURL#1#2{{\small{DOI}:
  \href{http://dx.doi.org/#2}{#1}}}

\end{document}